\documentclass[12pt,letterpaper, reqno]{amsart}

\usepackage{amssymb,latexsym, amsmath, amsxtra}
\usepackage{graphicx}

\theoremstyle{plain}
\newtheorem{theorem}{Theorem}[section]

\newtheorem{lemma}[theorem]{Lemma}

\theoremstyle{definition}

\theoremstyle{remark}

\numberwithin{equation}{section}
\numberwithin{theorem}{section}
\numberwithin{table}{section}
\numberwithin{figure}{section}

\newcommand{\R}{\mathbb R}
\newcommand{\Z}{\mathbb Z}
\newcommand{\Q}{\mathbb Q}

\def\({\left(}
\def\){\right)}

\begin{document}
\title{The second Dirichlet coefficient starts out negative}
\author{David~W.~Farmer and Sally~Koutsoliotas}
\dedicatory{Dedicated to the memory of Marvin Knopp.}

%\address{
%{\parskip 0pt
%American Institute of Mathematics\endgraf
%farmer@aimath.org\endgraf
%}
%  }

\begin{abstract}
Classical modular forms of small weight and low level are likely to have 
a negative second Fourier coefficient.
Similarly, the labeling scheme for
elliptic curves tends to give smaller labels to the higher-rank curves.
These observations are easily made when browsing the
L-functions and Modular Forms Database, available at
http://www.LMFDB.org/.  
An explanation lies in the L-functions associated
to these objects.
\end{abstract}

\maketitle

%newforms on GL(2)
% weight 2, level \le 21 (7 examples)
% weight 4, level \le 9 (5 examples)
% weight 6, level \le 4 (2 examples)
% weight 8, level \le 2 (1 example)
% weight 10, none (first is level 2, which has a2 positive
% weight 12, level \le 1  (1 example)

% elliptic curves of rank 1
% 37a(of 2) 43a 53a 57a(of 3) 58a(of 2) 61a 65a 77a(of 3) 79a 82a 83a 88a 89a(of 2)
% 91ab(2 of 2) 92a(of 2) 99a(of 4) 101a 102a(of 3)  [[fill in more data]]
% the first exception ia 121b.  Note that:
% 121a: q - q^2 + 2 q^3 ...
% 121b: q     - q^2 - 2 q^4 -3 q^5 ....

%elliptic curves of rank 2
%everything looks good up to 997, when you have the odd case of two
%rank 2 and one rank one.  Let's make an exception for that one.
% next exception is 1147 (one of rank 1 and one of rank 2)

% Siegel MF
% First several weights of Sp(4,Z), and first several paramodular levels of weight 2

% Hilbert MF /Q(\sqrt{5})
% Seem to have a(2) and a(3) vanish?

\section{Introduction} 

%The L-functions and Modular Forms DataBase (LMFDB), available at
%http://www.LMFDB.org/, provides a new approach  to

Modular forms and related objects have series expansions which typically
are normalized so that the initial coefficient is~$1$, making the second
coefficient the first instance containing useful information.
We present data which show that, for several families of objects,
the ``smallest'' members in the family tend to have a negative
second coefficient.  Here ``smallest'' refers to a natural ordering
of the family.  These observations seem surprising at first, because when averaged
over the family, the second coefficient has a known (or conjectured)
distribution \cite{CDF, Se, Sa} which is equally likely to be positive or negative.
The explanation lies in the fact that we can associate an L-function to each
object in the family, and we describe a
general principle which says that the
``first''  L-function in a family is likely to have a negative
second Dirichlet coeffcient.  That principle is a straightforward application
of Weil's explicit formula.

There is a substantial literature on bounding the
location of the first negative coefficient of modular forms and
other objects \cite{MR2503991,MR2352824,MR2232011,MR2726725,MR0696516}.
The situation we discuss here, in which a negative coefficient appears
as soon as possible, should be viewed as a transient phenomenon,
and not related to the general question of the oscillation of Fourier coefficients.

In Section~\ref{sec:data} we provide data on the first few
coefficients of
holomorphic cusp forms on Hecke congruence groups,
as well as
elliptic curves of higher rank.  In many cases,
the first nontrivial coefficient is negative.
%and Siegel modular forms.
In Section~\ref{sec:explicit} we make the argument that the
\hbox{L-functions} associated to the objects are 
responsible for
this phenomenon and we use the explicit formula to justify
that conclusion.

\thispagestyle{empty}

\section{The data}\label{sec:data}

\subsection{Classical holomorphic cusp forms}

Let $S_k(N)$ denote the space of holomorphic cusp forms of weight~$k$
for the Hecke congruence group~$\Gamma_0(N)\subset PSL(2,\Z)$, and let
$S_k^{new}(N)$ denote the subspace of newforms.  The newform space has
a distinguished basis of simultaneous eigenfunctions of the Hecke
operators and Atkin-Lehner operators.  The newforms have a Fourier
expansion
$$
f(z) = \sum_{n=1}^\infty a(n)e^{2\pi i n z},
$$
where $a(1)=1$.
A standard reference on this material is Iwaniec's book~\cite{iwaniec}.

\begin{table}
\[
\begin{array}{r|ccccccccccccccc}
N=&1&2&3&4&5&6&7&8&9&10&11&12&13&14&15\\
\hline
k=2&&&&&&&&&&&\bullet&&&\bullet&\bullet\\
4&&&&&\bullet&\bullet&\bullet&-&-&\circ&2&-&3&2&2\\
6&&&\bullet&-&\circ&\circ&3&1&1&3&4&&5&2&4\\
8&&\bullet&\circ&&3&1&3&2&3&1&6&2&7&4&4\\
10&&\circ&2&1&3&1&5&2&3&3&8&1&9&4&6\\
12&\bullet&&\circ&1&3&3&5&3&4&5&8&2&11&6&8\\
\end{array}
\]

\caption{\sf \label{fig:GL2}
The signs of the $2^{nd}$ Fourier coefficients of newforms
in $S_k^{new}(N)$.  The symbol $\bullet$ represents a form with a
negative coefficient, $\circ$ a form with a positive coefficient,
$-$ a form with a zero coefficient, and a positive integer is
the dimension of the newform space.
}
\end{table}

Table~\ref{fig:GL2} provides information about the dimension of the
space~$S_k^{new}(N)$ and the signs of the Fourier coefficient $a(2)$
of the newform basis.
From Table~\ref{fig:GL2}, one can see that if $k$ and $N$ are both
small, then $S_k^{new}(N)$ is trivial, and the dimension of
$S_k^{new}(N)$ grows regularly as a function of $k$, and irregularly
as a function of~$N$.  These facts are well known.  But what was
previously not observed, and seems surprising at first, is that for most
of the newforms which are on the border of the non-zero spaces,
$a(2)$ is negative.  
The negativity of these coefficients is not a coincidence: they arise
from a simple principle involving L-functions which we describe
in Section~\ref{sec:Lfunctions}.
%This is even more surprising when
%one knows that the set of all $a(2)$s has a distribution
%which is equally often positive positive or negative~\cite{Ser,CF}.
%In Section~\ref{sec:SkNnew} we explain why one expects $a(2)$ to
%be negative when the weight $k$ and level $N$ are small.

While not included in Table~\ref{fig:GL2},
we note that the second coefficient is negative for
$f\in S_2^{new}(N)$ for all $N\le 21$.
That is, there are four more $\bullet$ along the top row before
the first $\circ$ appears.
While this may also seem
surprising, it turns out to have a similar explanation,
which we give in Section~\ref{sec:small weight}.

%\subsection{Siegel modular forms}
%
%A natural generalization of holomorphic cusp forms on $\GL(2)$
%is (holomorphic) Siegel modular forms on $\Sp(4)$.
%Available computations of these function is limited, but much
%of what has been done is available in the LMFDB.  Specifically,
%the LMFDB allows one to browse Siegel modular forms of level~1
%and low weight, and prime level and weight~2.  For the forms which
%are available as of the time of this writing, in Table~\ref{tab:siegelTp}
%we collect the eigenvalues of the Hecke operators $T(2)$ and $T(2)$.
%
%[[
%
%Table goes here
%
%]]
%
%As can be seen in the table, all of those eigenvalues are negative.
%It is conjectured (but has not yet been proven) that each of those
%eigenvalues has a distribution, which is equally likely to be
%positive or negative.

\subsection{Elliptic curves}
Tables of elliptic curves$/\Q$ have both inspired and benefitted
from significant theoretical work.  As of this writing, all elliptic
curves of conductor less than 350,000 have been tabulated
by John Cremona, and detailed
information about them is available in the 
L-functions and Modular Forms Database (LMFDB)~\cite{lmfdb}.

An issue 
is how to label elliptic curves.  The goal is staightforward:
each curve should be given a label which specifies its conductor,
its isogeny class, and the isomorphism class within the isogeny class.
The conductor is an integer, so there is no ambiguity.  
The order in which to list the isogeny classes, however, requires a choice,
as does the order for listing the curves within an isogeny class.

Various methods for ordering the isogeny classes have been used, and
until recently the ``Cremona label'' has been the standard.
Unfortunately, the Cremona label is difficult to describe, and in some
cases cannot be derived from first principles.  See~\cite{Crem}
for details.

These shortcomings have been addressed by a
new labeling scheme, also developed by John Cremona, 
referred to as the
``LMFDB label.''  For example, the elliptic curve
$$
E\ :\ y^2 = x^3 + x^2 + 210 x + 1764
$$
has LMFDB label  ``672.e2''.  
That specific label means:
\begin{itemize}
\item{} $E$ has conductor 672,
\item{}  when the newforms with rational integer coefficients
in $S_2^{new}(672)$ are ordered lexicographically by their
Fourier coefficients, the form associated to $E$ appears 5th on the
list (because ``e'' is the 5th letter), and 
\item{}  when the elliptic curves in the isogeny class of $E$
are put in minimal Weierstrass form
$$
y^2+a_1xy+a_3y=x^3+a_2 x^2+a_4x+a_6,
$$
then $E$ appears 2nd when the $[a_1,a_2, a_3, a_4, a_6]$ are
ordered lexicographically.
\end{itemize}
Note that the modularity \cite{BCDT}  of ellipic curves$/\Q$ is necessary to 
ensure that each curve has an LMFDB label.

One of the interesting properties of an elliptic curve
is its \emph{rank}, which is the rank of the
Mordell-Weil group $E(\Q)$ of rational points on $E$.  The
rank of an elliptic curve is the subject of the 
Birch and Swinnerton-Dyer conjecture, which equates the rank with the order
of vanishing of the L-function $L(s,E)$ at the critical point $s=\frac12$.

Table~\ref{tab:ECrank} shows the LMFDB label for the isogeny
classes of the first
few elliptic curves of rank~1, along with
the number of isogeny classes with that conductor.
Note that
we give the label of an isogeny class, not an elliptic curve,
because all the curves in an isogeny class have the same
rank.

\begin{table}
\[
\begin{array}{c|ccccccccccccccc}
\text{class}&37.a&43.a&53.a&57.a&58.a&61.a&65.a&77.a&79.a&82.a&88.a&89.a\\
\hline
\#N&2&1&1&3&2&1&1&3&1&1&1&2\\
\end{array}
\]
\phantom{x}
\[
\begin{array}{cccccccccccccccc}
%\multicolumn{9}{c}{\mathstrut} \\
91.a \text{ and } 91.b&92.a&99.a&101.a&102.a&106.a&112.a&117.a&118.a&121.b\\
\hline
2&2&4&1&3&4&3&1&4&3\\
%\multicolumn{9}{c}{\mathstrut} \\
\end{array}
\]
%\[
%\begin{array}{cccccccccccccccc}
%102.a&106.a&112.a&117.a&118.a&121.b\\
%\hline
%3&4&3&1&4&3\\
%\end{array}
%\]

\caption{\sf \label{tab:ECrank}
The LMFDB labels of the isogeny classes of the first 23
elliptic curves of rank 1, and the number of isogeny classes
$\#N$
with conductor~$N$.
}
\end{table}

For the entries in Table~\ref{tab:ECrank} where there is a
single isogeny class of a given conductor, then of course that
class has to be given the label ``$a$.''
But what seems surprising at first is that even when there
are other isogeny classes, the entries in Table~\ref{tab:ECrank}
almost all have the label ``$a$.''
In fact, the first 11 times there was a rank 1 curve and also another
isogeny class with that same conductor, the rank 1 isogeny
class is given the label ``$a$.''  That seems highly unlikely, considering
that the labeling scheme does not appear to have any reference
to the rank.

In Table~\ref{tab:ECrank2} we give the labels and the number of
isogeny classes with a given conductor, for the first few elliptic
curves of rank 2.  We have omitted those cases where there is only
one isogeny class with that conductor, such as the first case
of rank~2: $389.a$.

\begin{table}
\[
\begin{array}{c|ccccccccccccccc}
\text{class}&446.a&571.a&664.a&681.a&718.a&794.a&817.a&916.a&994.a\\
\hline
\#N&4&2&3&5&3&4&2&5&11\\
\end{array}
\]

\caption{\sf \label{tab:ECrank2}
The LMFDB labels of the isogeny classes of the first 9
elliptic curves of rank~2, and the number of isogeny classes
$\#N$ with conductor $N$,
for those cases where there is more than one
isogeny class for that conductor.
}
\end{table}

As in the first 11 entries of Table~\ref{tab:ECrank},
every curve in Table~\ref{tab:ECrank2} has isogeny
class ``$a$.''  In fact, this phenomenon for rank 2
continues until conductor 1147, where isogeny class ``$b$''
has rank~2.
A similar phenomenon occurs for rank~3, which the interested
reader can explore in the LMFDB.
We will see that this has a similar explanation to our previous
observations.

\section{L-functions and the explicit formula}\label{sec:explicit}

All the observations in Section~\ref{sec:data} have the same underlying
cause: all are consequences of the following general
principle:

\begin{center}
\emph{L-functions which barely exist tend to have negative coefficients.}
\end{center}

In this section we describe what we mean by ``L-function'' and
the sense in which an \hbox{L-function} can ``barely exist.''
We then show how the above principle explains our observations.

%L-functions are part of number theory, not analysis.  (This simple
%fact allows one to immediately discard many claimed proofs of the
%Riemann Hypothesis.)  Thus, assertions like ``$N > 1/1000$'' are
%not particularly  interesting, because we already know that
%$N$ has to be a positive integer.

\subsection{$L$-functions}\label{sec:Lfunctions}

By an \emph{$L$-function}, we mean 
a Dirichlet series with a functional equation and an Euler product.
%the $L$-function attached to
%an irreducible unitary cuspidal automorphic representation of $GL_n$ over $\Q$.
Furthermore, we assume the Ramanujan-Petersson conjecture and the
Generalized Riemann Hypothesis.  This means that we can write the
$L$-function as
a Dirichlet series
\begin{equation}\label{eqn:ds}
L(s) = \sum_{n=1}^\infty \frac{a(n)}{n^s}
\end{equation}
where $a(n)\ll n^\delta$ for any $\delta>0$, which has an Euler
product
\begin{equation}\label{eqn:ep}
L(s)=\prod_p L_p(p^{-s})^{-1},
\end{equation}
where $L_p$ is a polynomial, and the product is over the primes.
As is usual in the theory of L-functions, $s=\sigma + i t$ is a
complex variable.
We assume the Dirichlet coefficients are real,
so the L-function satisfies 
a functional equation which can be written in the form
\begin{equation}\label{eqn:fe}
    \Lambda(s) =  Q^s \prod_{j=1}^d \Gamma\left(\frac{s}{2} + \mu_j\right) L(s) 
= \varepsilon {\Lambda(1 - {s})}.
\end{equation}
Here $\varepsilon=\pm 1$ and we assume that $\mu_j\ge 0$ and~$Q> 0$.
%The normalized $\Gamma$-function is defined as
%\begin{equation}
%\Gamma_\R(s) = \pi^{-s/2}\Gamma(s/2),
%\end{equation}
%where $\Gamma(s)$ is the usual Euler Gamma function.
The number $d$ is called the \emph{degree} of the $L$-function, which
for all but finitely many $p$ is also the degree of the polynomial~$L_p$.
The L-functions associated to the arithmetic objects considered in this
paper are conjectured to satisfy the 
Ramanujan-Petersson bound, which asserts that the polynomials $L_p$ have
all their zeros on or outside the unit circle, so the 
Dirichlet coefficients actually satisfy the more precise bound
\begin{equation}\label{eqn:rama}
| a(p) | \le d,
\end{equation}
for $p$ prime.
% In particular, $|a(p)|\le d$ for prime~$p$.
Also, the L-functions considered in this paper are conjectured to
satisfy the analogue of the Riemann Hypothesis, which says that
the zeros of the L-function in the strip $0<\sigma<1$ have the
form $\rho = \frac12 + i \gamma$ with $\gamma\in\R$.

A simple but powerful tool in the study of L-functions is
Weil's explicit formula.

\begin{lemma}\label{lem:weil} Suppose that $L(s)$ has a Dirichlet series
expansion
\eqref{eqn:ds} which continues to an entire function such that
\begin{equation}
    \Lambda(s) =
		Q^s \prod_{j=1}^d \Gamma\left(\frac{s}{2} + \mu_j\right) L(s)
%	= \varepsilon \overline{\Lambda(1-\overline{s})}
	= \varepsilon {\Lambda(1-{s})}
\end{equation}
is entire and satisfies the mild growth condition
$L(\sigma + it) \ll |t^A|$ for some $A>0$, uniformly in $t$ for bounded~$\sigma$.
Let $f(s)$ be an even function which is 
holomorphic in a horizontal strip $-(1/2 + \delta) < Im(s) < 1/2 + \delta$ with $f(s) \ll \min(1, |s|^{-(1+\epsilon)})$
in this region, and suppose that $f(x)$ is real-valued for real $x$.
Suppose also that the Fourier transform
of $f$, defined by
\[
    \hat f(x) = \int_{-\infty}^\infty f(u)e^{-2\pi i u x} dx,
\]
is such that
\[
    \sum_{n=1}^{\infty} \frac{c(n)}{n^{1/2}} \hat{f} \left( \frac{\log{n}}{2 \pi} \right) 
% + \frac{\overline{c(n)}}{n^{1/2}} \hat f\left( -\frac{\log n}{2 \pi}\right)
\]
converges absolutely, where $c(n)$ is defined by
\begin{equation}
\frac{L'}{L}(s) = \sum_{n=1}^{\infty} \frac{c(n)}{n^s} .
\end{equation}
Then
\begin{align} \label{weil}
\sum_{\gamma} f(\gamma) =\mathstrut &  \frac{\widehat{f}(0)}{\pi} \log{Q} + \frac{1}{2 \pi}
\sum_{j=1}^d \ell(\mu_j, f) %\cr
%&+ %\frac{1}{2 \pi} 
+ \frac{1}{\pi}
\sum_{n=1}^{\infty} \frac{c(n)}{n^{1/2}} \hat{f} \left( \frac{\log{n}}{2 \pi} \right) ,
% + \frac{\overline{c(n)}}{n^{1/2}} \hat f\left( -\frac{\log n}{2 \pi}\right)
\end{align}
where
\begin{equation}\label{ellterm}
\ell(\mu, f) = Re\ \left\{\int_\R \frac{\Gamma'}{\Gamma} \left( \frac{1}{2} \left( \frac{1}{2} + i t \right) + \mu \right) f(t) dt\right\} - \hat f(0)\log \pi ,
\end{equation}
and the sum $\sum_\gamma$ runs over all non-trivial zeroes of $L(s)$.
\end{lemma}

\begin{proof}
This can be found in Iwaniec and Kowalski \cite{IK}, page~109, but note that they
use a different
normalization for the Fourier transform.
\end{proof}

The following is a precise version of the principle stated at the beginning of this section.

\begin{theorem}\label{thm:meta}  Fix non-negative real numbers $\mu_1,\ldots,\mu_d$ in \eqref{eqn:fe}.
There exist real numbers $0<Q_0<Q_1$ such that if $0<Q<Q_0$ then
there do not exist any L-functions with functional equation~\eqref{eqn:fe}
satisfying the Ramanujan bound~\eqref{eqn:rama} and the Riemann Hypothesis.
And if $Q_0 < Q < Q_1$ then any L-function satisfying those three conditions
must have $a(2) < 0$.
\end{theorem}

In other words, if an L-function is just barely able to exist, meaning that 
the parameter $Q$ in its functional equation is only slightly larger
than the minimum threshold imposed by the $\mu_j$, then it must have
a negative second Dirichlet coefficient.

The proof involves choosing an appropriate test function in the explicit formula.
We will make a simple choice in order to illustrate the method, making
no attempt to obtain optimal results.  
%As an exercise the reader may wish to confirm
%that the majority (but not all) of the negative coefficients listed in our
%introduction do indeed fall under the scope of this theorem.

\begin{proof}[Proof of Theorem~\ref{thm:meta}]
In the explicit formula, choose the function
\begin{align}
f(x) &= \frac{1}{2\pi} \frac{\sin^2(x/2)}{(x/2)^2}, \\
\intertext{which satisfies}
\hat{f}(x) &= 
\begin{cases} 1-2 \pi |x| & \text{if} -\tfrac{1}{2\pi}<x<\tfrac{1}{2\pi} \cr
              0 & \text{otherwise}.
\end{cases}
\end{align}
The key properties we require are that $f$ is non-negative, and the
support of $\hat{f}$ contains $\log(2)/(2\pi)$ but not $\log(n)/(2\pi)$
for any integer $n>2$, and $\hat{f}(\log(2)/(2\pi))>0$.

Substituting into \eqref{weil} we obtain
\begin{align} \label{weil2}
\sum_{\gamma} f(\gamma) =\mathstrut &  %\frac{\widehat{f}(0)}
\frac{1}{\pi} \log{Q} + \frac{1}{2 \pi}
\sum_{j=1}^d \ell(\mu_j, f)
+ %\frac{1}{2 \pi}
0.069066
\, c(2).
%\frac{1}{\pi}
%\sum_{n=1}^{\infty} \frac{c(n)}{n^{1/2}} \hat{f} \left( \frac{\log{n}}{2 \pi} \right)
%% + \frac{\overline{c(n)}}{n^{1/2}} \hat f\left( -\frac{\log n}{2 \pi}\right)
\end{align}

The sum on the left side is non-negative, and the sum over $\ell(\mu_j, f)$ is a
constant depending on the $\mu_j$ and our choice of $f$.
The Ramanujan bound on the Dirichlet coefficients gives us a bound on $c(2)$, coming from the
simple fact that for $p$ prime, $c(p) = - a(p) \log(p)$ where $a(p)$ is the $p$th Dirichlet coefficient
of the L-function.

Thus, if $Q$ is sufficiently small then the right side must be negative, which
is a contradiction.  That proves the
existence of the number $Q_0$.  And if $Q$ is slightly larger, then the only way for the
right side to be non-negative is if $0.069066 \, c(2)$ is positive; this is
equivalent to $a(2)$ being negative.  That proves the
existence of the number $Q_1$.
\end{proof}

To complete our explanation of the prevalence of $a(2)<0$ for cusp forms on the
edge of existence, we note that for the L-functions associated to
$F\in S_k^{new}(N)$, the parameters $Q$ and $\mu$ in the functional equation are given by
$ Q= N/\pi$ and $\{\mu_1,\mu_2\} = \{\frac{k-1}{4},\frac{k+1}{4}\}$.
Applying Theorem~\ref{thm:meta} with $k$ fixed, we see that if $N$ is sufficiently
small then $S_k^{new}(N)$ must be empty, and if $N$ is slightly larger and
$S_k^{new}(N)$ is nonempty, then $a(2)<0$.  

%  
%and for Siegel paramodular forms of level $N$ and weight~$k$ we have
%$Q=N/\pi^2$ and $\{\mu,\mu,\mu,\mu\} = \{\frac14,\frac34, \frac{2k-3}{4},\frac{2k-1}{4}\}$.  

For the elliptic curves of higher rank, we note that, according to the Birch and
Swinnerton-Dyer conjecture, the L-function of a rank $r$ elliptic curve
has an order $r$ zero at the critical point.
%  In our analytic normalization,
%which has a functional equation relating $s$ to $1-s$, that corresponds
This translates 
to $r$ zeros with $\gamma=0$ in~\eqref{weil}.
Those $\gamma=0$ terms add a large positive contribution to
the left side of \eqref{weil}, equivalently \eqref{weil2}.  Thus, the lower bound on
possible levels $N=\pi Q$ for a rank $r$ elliptic curve is an increasing
function of $r$. And just as in the proof of Theorem~\ref{thm:meta}, those
$Q$ which are slightly larger than the minimum must be accompanied by a negative
value for $a(2)$.  Finally, since the isogeny classes of elliptic curves are
ordered lexicographically by the L-function (equivalently, modular form)
coefficients, those with a negative $a(2)$ are more likely to be listed first,
receiving ``a'' as the label of their isogeny class.

\subsection{Small weight}\label{sec:small weight}

For the larger weights in Table~\ref{fig:GL2}, the $a(2)<0$ phenomenon only persists
for a small strip around the edge of the region where the cusp forms
are able to exist.  However, for weight $k=2$ we find that $a(2)<0$
for all $N\le 21$.  This also has a simple explanation coming from the
explicit formula.

For a weight $k$ cusp form, the parameters $\{\mu_1, \mu_2\}$ in the
functional equation are $\{\tfrac{k-1}{4},\tfrac{k+1}{4}\}$.
In Figure~\ref{fig:digamma}
we plot the factor
\begin{equation}\label{eqn:digamma}
Re\left[  \frac{\Gamma'}{\Gamma}\left(\frac12\left(\frac12+ i t \right)+\mu\right) \right]
\end{equation}
which occurs in the $\ell(\mu,f)$ term in the explicit formula,
\eqref{ellterm}, for various $\mu$.

\begin{figure}[htp]
\begin{center}
\scalebox{1}[0.8]{\includegraphics{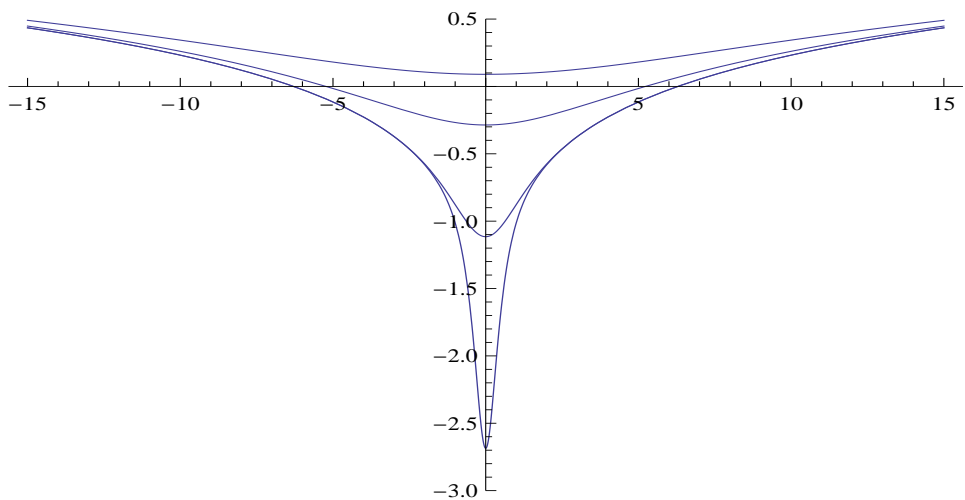}}

\caption{\sf
Plots of the real part of $\frac{\Gamma'}{\Gamma}\left(\frac12\left(\frac12+ i t \right)+\mu\right)$
for $\mu = 0$, $1$, $4$, and $8$, where smaller values of $\mu$
correspond to lower graphs.
} \label{fig:digamma}
\end{center}
\end{figure}

In the integrand of \eqref{ellterm},
the function in Figure~\ref{fig:digamma} is multiplied by $f(t)$, which is positive and has most of
its support near $t=0$. Therefore the integrand is mostly negative when
the weight $k$ is small, so the contribution of $\ell(\mu,f)$ to the right
side of the explicit
formula will be negative.  Thus, for small $k$ the value of $Q = \pi N$ must be larger 
in order 
for the right side of the explicit formula to be non-negative.
%to enable the
%L-function to exist.  
And as before, for $Q$ slightly above the threshold value
it is necessary for $a(2)$ to be negative.  
Let $-\delta$ be the negative amount which $a(2)$ could contribute to the right side.
Since it is $\log Q$ that contributes
to the right side, 
one needs 
$\log Q$ to increase by $\delta$ before it is possible for $a(2)$ to be positive.
But if $Q$ is larger, as it must be if $k$ is small, then $\log Q$ 
increases more slowly, so one requires a larger increase in $Q$ to increase $\log Q$ by~$\delta$.
This suggests that if $k$ is small, once $Q$ is large enough for the curve (or the L-function) to possibly exist,
there is a wider range of $Q$ values which require $a(2)$ to be negative.
That explains the large number of $\bullet$ at the top of
Table~\ref{fig:GL2}.

It is a curious phenomenon that elliptic curves of a given rank tend to appear
with conductors that are only slightly larger (on a logarithmic scale) than the minimum forced 
upon them
by the explicit formula.  This means that many of the initial coefficients,
and not just $a(2)$, must be negative.  This fact imparts some surprising properties
on the L-functions.  See Michael Rubinstein's paper~\cite{Rub} and its references for further
discussion, data, and plots.

%\subsection{Hyperelliptic curves}
%As noted above, the phenomenon of $a_2<0$ is more prominent
%when the weight is small,
%meaning that it persists across a wider range of levels.
%A particular example is hyperelliptic curves.  Data for hyperelliptic
%curves were recently added to the LMFDB, but the calculations used in~\cite{FKLvv,FKLd4w0}
%find that the first $M$ hyperelliptic curves, which have conductors in the range
%$169\le N \le ??$, all have L-functions with a negative second Dirichlet coefficient.

\end{document}